\documentclass[10pt]{article}
\usepackage{hyperref}
\usepackage{amsmath}
\usepackage[dvips]{graphicx}
\usepackage{amsthm}
\usepackage{epsfig}
\usepackage{amssymb}
\usepackage{dsfont}

\newtheorem{theorem}{Theorem}[section]

\newtheorem{corollary}[theorem]{Corollary}

\begin{document}

\begin{center}
{\Large  Qualitative properties of positive solutions of quasilinear
equations with Hardy terms}
\end{center}
\vskip 5mm

\begin{center}
{\sc Yutian Lei} \\
\vskip2mm
Institute of Mathematics\\
School of Mathematical Sciences\\
Nanjing Normal University\\
Nanjing, 210023, China \\
Email:leiyutian@njnu.edu.cn
\end{center}

\vskip 5mm {\leftskip5mm\rightskip5mm \normalsize
\noindent{\bf{Abstract}} In this paper, we are concerned with the
quasilinear PDE with weight
$$
-div A(x,\nabla u)=|x|^a u^q(x), \quad u>0 \quad \textrm{in} \quad
R^n,
$$
where $n \geq 3$, $q>p-1$ with $p \in (1,2]$ and $a \in (-n,0]$. The positive weak solution
$u$ of the quasilinear PDE is $\mathcal{A}$-superharmonic
and satisfies $\inf_{R^n}u=0$. We can introduce an
integral equation involving the wolff potential
$$
u(x)=R(x) W_{\beta,p}(|y|^au^q(y))(x), \quad u>0 \quad \textrm{in}
\quad R^n,
$$
which the positive solution $u$ of the quasilinear PDE satisfies.
Here $p \in (1,2]$, $q>p-1$, $\beta>0$ and $0 \leq -a<p\beta<n$. When $0<q
\leq \frac{(n+a)(p-1)}{n-p\beta}$, there does not exist any positive
solution to this integral equation. When $q>\frac{(n+a)(p-1)}{n-p\beta}$, the positive
solution $u$ of the integral equation is bounded and decays
with the fast rate $\frac{n-p\beta}{p-1}$ if and only if it is integrable
(i.e. it belongs to $L^{\frac{n(q-p+1)}{p\beta+a}}(R^n)$).
On the other hand, if the bounded solution
is not integrable and decays with some rate, then the rate must be
the slow one $\frac{p\beta+a}{q-p+1}$. Thus, all
the properties above are still true for the quasilinear PDE.
Finally, several qualitative properties for this PDE are discussed.
}

\newtheorem{proposition}[theorem]{Proposition}

\renewcommand{\theequation}{\thesection.\arabic{equation}}
\catcode`@=11
\@addtoreset{equation}{section}
\catcode`@=12

\section{Introduction }

In this paper, we are concerned with positive solutions of the
following quasilinear equation with a Hardy term
\begin{equation} \label{1}
    -div A(x,\nabla u)=|x|^a u^q(x), \quad u>0 ~ in ~ R^n,
\end{equation}
Here $n \geq 3$, $q>p-1$ with $p \in (1,2]$, $-a \in [0,n)$ and
$A:R^n \times R^n \to R^n$ is a vector valued mapping
satisfying:

1. the mapping $x\to A(x,\xi)$ is measurable for all $\xi \in R^n$;

2. the mapping $\xi\to A(x,\xi)$ is continuous for a.e. $x \in
R^n$.

In addition, there are constants $0 < \mu_1 \leq \mu_2 < \infty$
such that for a.e. $x \in R^n$, and for all $\xi \in R^n$,
\begin{equation} \label{2}
\left \{
    \begin{array}{ll}
    &A(x,\xi) \cdot \xi \geq \mu_1|\xi|^p, \quad A(x,\xi) \leq \mu_2|\xi|^{p-1}; \\
    &[A(x,\xi_1)-A(x,\xi_2)] \cdot (\xi_1-\xi_2) > 0, \quad if ~ \xi_1 \neq \xi_2; \\
    &A(x,\lambda\xi)=\lambda|\lambda|^{p-2}A(x,\xi), \quad if ~ \lambda
\neq 0.
    \end{array}
    \right.
\end{equation}

A function $u \in W_{loc}^{1,p}(R^n) \cap C(R^n)$ is called the
{\it weak solution} of (\ref{1}), if
\begin{equation} \label{weaksolution}
\int_{R^n} A(x,\nabla u) \nabla \zeta dx =\int_{R^n}|x|^a u^q
\zeta dx, \quad \forall \zeta \in C_0^\infty(R^n).
\end{equation}

In the special case $A(x,\xi)=|\xi|^{p-2}\xi$, $div A(x,\nabla u)$
is the usual p-Laplacian defined by $div(|\nabla u|^{p-2}\nabla
u)$. Now, (\ref{1}) becomes
\begin{equation} \label{PDE}
-div(|\nabla u(x)|^{p-2}\nabla u(x))=|x|^a u^q(x), \quad u>0 \quad
in \quad R^n.
\end{equation}
This equation arose in many fields such as nonlinear functional analysis,
astrophysics and astronomy (cf. \cite{BP}, \cite{BW} and
\cite{CR}). In particular, it is essential in the study of the extremal function
of the Hardy-Sobolev type inequality \cite{BT}
\begin{equation} \label{HS}
\Lambda(\int_{R^n}|u|^{q+1}|x|^{a}dx)^{\frac{1}{q+1}} \leq
(\int_{R^n}|\nabla u|^pdx)^{1/p}, \quad
\forall u \in \mathcal{D}^{1,p}(R^n),
\end{equation}
where $0 \leq -a<p$, $1 \leq p<q+1=\frac{p(n+a)}{n-p}$, and
$\mathcal{D}^{1,p}(R^n)$ is the homogeneous Sobolev space. We call
$q=\frac{p(n+a)}{n-p}-1$ the {\it critical exponent}. Such an
inequality is the special case of the Cafarelli-Kohn-Nirenberg
inequality (cf. \cite{CKN}, \cite{CW} and \cite{Lin}). Applying the
symmetrization and the theories of ODE, one can obtain the sharp
constant $\Lambda$ (cf. \cite{L}, \cite{PV1} and \cite{SSW}). To
find the corresponding extremal function, we consider the
minimization problem
$$
\Lambda=\inf\{\|\nabla u\|_p^p;u \in \mathcal{D}^{1,p}(R^n),
\int_{R^N}|u|^{q+1}|x|^adx=1\}.
$$
Since this minimization problem is invariant under the scaling transformation,
the variational methods are difficult to be used. Badiale and Tarantello 
found the solution by the concentration compactness
principle (cf. \cite{BT}). To describe the shape of the extremal functions, we
investigate the Euler-Lagrange equation (\ref{PDE}). In \cite{MS}, it was proved
that all the extremal functions are cylindrically
symmetric.

When $p=2$, (\ref{PDE}) becomes
\begin{equation} \label{dhua}
-\Delta u=|x|^a u^q, \quad
u>0 \quad in \quad R^n.
\end{equation}
Phan and Souplet \cite{PS} studied the existence of the positive solution and
obtained the Liouville type results. Recently, \cite{Lei-JDE} used
an equivalent integral equation to obtain the decay rates of the
positive solutions when $|x| \to \infty$. If $a=0$, it is
associated with the study of the well known Lane-Emden equation
$$
-\Delta u=u^q, \quad u>0 \quad in \quad R^n,
$$
which has been well studied (cf. \cite{CGS}, \cite{CL1},
\cite{GNN} and \cite{Li}).

When $p \neq 2$, it is difficult to find an equivalent integral
equation. If a positive solution $u$ is a
$\mathcal{A}$-superharmonic function and satisfies $\inf_{R^n}
u=0$, then $u$ solves another integral equation involving the
Wolff potential
\begin{equation} \label{IE}
u(x)=R(x) W_{\beta,p}(|y|^au^q(y))(x)
\end{equation}
with $\beta=1$ (cf. \S4). Here $R(x)$ is double bounded. Namely,
there exists $C>0$ such that $\frac{1}{C} \leq R(x) \leq C$ for
all $x \in R^n$. The definition of the $\mathcal{A}$-superharmonic
function can be found in \cite{KM} and \cite{PV}.

The Wolff potential of a positive function $f \in L_{loc}^1(R^n)$
is defined as (cf. \cite{HW})
$$
W_{\beta,p}(f)(x) =\int_0^\infty[\frac{\int_{B_t(x)}
f(y)dy}{t^{n-p\beta}}] ^{\frac{1}{p-1}}\frac{dt}{t},
$$
where $p>1$, $\beta > 0$, $p\beta < n$, and $B_t(x)$ is a ball of
radius $t$ centered at $x$. This potential can help us to
understand many nonlinear problems (see \cite{COV}, \cite{KKT},
\cite{KM1}, \cite{KM}, \cite{La}, \cite{Maly}, \cite{Min} and
\cite{PV}).

When $a=0$ and $q$ is the critical exponent $\frac{np}{n-p}-1$,
Ma, Chen and Li \cite{ChLM} obtained the integrability, boundedness and
the Lipschitz continuity of positive solutions of
(\ref{IE}). Based on these results, paper
\cite{LeiLi} estimated the fast decay rate. This asymptotic result is
also true for the Wolff type integral system (cf.
\cite{SL}). Moreover, if $R(x) \equiv 1$, those positive solutions
are radially symmetric and decreasing about $x_0 \in R^n$ (cf.
\cite{ChenLi}).

Furthermore, if $\beta=\alpha/2$ and $p=2$, (\ref{IE})
(with $R(x) \equiv 1$) is reduced to
an integral equation involving the Riesz potential
$$
u(x)=\int_{R^n}\frac{|y|^au^q(y)dy}{|x-y|^{n-\alpha}},
\quad u>0 \quad in \quad R^n.
$$
Lu and Zhu \cite{LZ} obtained the radial symmetry and the
regularity of weak solutions. Moreover, if $\alpha=2$, this integral
equation is reduced to (\ref{dhua}). Mancini, Fabbri and Sandeep \cite{MFS} studied
the sharp constant and classified the extremal functions of
(\ref{HS}). If $a=0$, the integral equation above becomes
\begin{equation} \label{riesz-hls}
u(x)=\int_{R^n}\frac{u^q(y)dy}{|x-y|^{n-\alpha}}.
\end{equation}
This equation can be used to describe the extremal functions of the
Hardy-Littlewood-Sobolev inequality (cf. \cite{CLO1}, \cite{YLi}
and \cite{L}).
The radial symmetry of integrable solutions was proved by the method of moving
planes in integral forms. Using the regularity lifting lemma by the
contraction operators, Jin and Li \cite{JL} obtained the optimal
integrability. Afterwards, \cite{LLM-CV} presented the fast decay
rates.

In this paper, we consider that $q$ is not only the critical
exponent $\frac{p(n+a)}{n-p}-1$, but also the supercritical and
the subcritical cases. Let
\begin{equation} \label{ASSUM}
n \geq 3, \beta>0, p \in
(1,2], q>p-1, 0 \leq -a<p\beta<n.
\end{equation}
Write $p^*=\frac{np}{n-p\beta}$. When $\beta=1$, $p^*$ is the
Sobolev conjugate index, and the weak solution $u \in
\mathcal{D}^{1,p}(R^n)$ belongs to $L^{p^*}(R^n)$.

Introduce an important index
$$
s_0:=\frac{n(q-p+1)}{p\beta+a}.
$$
This index $s_0$ is closely related to the invariant of the equation and the
energy under the scaling transformation (see Theorem
\ref{prop6.1}). In addition, $s_0=p^*$ if and only if $q$ is the
critical exponent $q=\frac{p(n+a)}{n-p\beta}-1$. When $a=0$, this
critical exponent plays an important role in studying the
existence of positive solutions of (\ref{PDE}) (cf. Corollary II
in \cite{SZ}). When $q<\frac{p(n+a)}{n-p\beta}-1$, (\ref{PDE}) has
not any regular solution. When $q \geq \frac{p(n+a)}{n-p\beta}-1$,
(\ref{PDE}) has positive solutions. Papers \cite{Fr} and
\cite{KYY} estimated the decay rates of solutions. Moreover,
\cite{LeiLi} obtained the fast decay rates of the
$L^{p^*}(R^n)$-solutions of (\ref{IE}) when
$q=\frac{p(n+a)}{n-p\beta}-1$.

In section 2, we have a Liouville type theorem (see Theorem
\ref{th2.1}). In addition, we prove the following decay estimates
when $|x| \to \infty$ in sections 2 and 3.

\begin{theorem} \label{th1.2}
If $q \in (0,\frac{(n+a)(p-1)}{n-p\beta}]$, then (\ref{IE}) does
not have any positive solution for any double bounded function $R(x)$.
Assume $u$ is a positive
solution of (\ref{IE}) with (\ref{ASSUM}). Then
\begin{enumerate}
\item $u$ is bounded and decays with the fast rate
$\frac{n-p\beta}{p-1}$ if and only if $u \in L^{s_0}(R^n)$.

\item If a bounded solution $u \not\in L^{s_0}(R^n)$ decays with some rate, then the
rate must be the slow one $\frac{p\beta+a}{q-p+1}$.
\end{enumerate}
\end{theorem}

Consider the PDEs (\ref{1}) and (\ref{PDE}). We have the following
qualitative results which are proved in section 4.

\begin{theorem} \label{th1.1}
Let $u$ be a positive weak solution of (\ref{1}) with
(\ref{ASSUM}). Then
\begin{enumerate}
\item the results in Theorem \ref{th1.2} are still true. (Now,
$\beta=1$.) Furthermore, if $u \in L^{s_0}(R^n)$, then $u \in
\mathcal{D}^{1,p}(R^n)$.

\item Assume the weak solution $u \in L^{p^*}(R^n)$. Then $\nabla u
\in L^p(R^n)$ if and only if $|x|^au^{q+1} \in L^1(R^n)$.

\item Let $\lambda \neq 0$. The scaling function
$u_\lambda(x):=\lambda^\theta u(\lambda x)$ is still a weak solution
of (\ref{1}) if and only if $\theta$ is the slow rate
$\frac{p\beta+a}{q-p+1}$. Moreover,
$\|u_\lambda\|_\eta=\|u\|_\eta$ if and only if $\eta=s_0$.
\end{enumerate}
\end{theorem}

\begin{corollary} \label{th1.3}
If $u$ is a classical solution of
(\ref{PDE}) with (\ref{ASSUM}), then the following three items
are equivalent:
\begin{enumerate}
\item $u \in L^{s_0}(R^n)$;
\item $u$ is bounded and decays with the fast rate $\frac{n-p}{p-1}$ when $|x| \to \infty$;
\item $u \in \mathcal{D}^{1,p}(R^n)$.
\end{enumerate}
\end{corollary}

Finally, we shows that the weak solution of (\ref{1})
cannot be defined in $W^{1,p}(R^n)$ when $p \in [\sqrt{n},2]$.

\section{Nonexistence and slow decay rate}

In this section, we discuss the slow decay of positive
solutions of (\ref{IE}). In order to estimate the decay rate, we
first show that the exponent $q$ is larger than $\frac{(n+a)(p-1)}{n-p\beta}$.

\subsection{Nonexistence}

\begin{theorem} \label{th2.1}
If $q \in (0,\frac{(n+a)(p-1)}{n-p\beta}]$, then (\ref{IE}) does not have any
positive solution for any double bounded function $R(x)$.
\end{theorem}

\begin{proof}

{\it Step 1.}
Let
\begin{equation} \label{2.10*}
0<q<\frac{(n+a)(p-1)}{n-p\beta}.
\end{equation}
Suppose that $u$ solves (\ref{IE}), then we deduce a contradiction.

{\it Substep 1.1.}
From (\ref{IE}), for $|x|>1$ we can get
\begin{equation} \label{low-ini}
u(x) \geq c\int_{2|x|}^\infty t^{\frac{\beta p-n}{p-1}} \frac{dt}{t}
=\frac{c}{|x|^{\frac{n-\beta p}{p-1}}}:=\frac{c}{|x|^{a_0}},
\end{equation}
since $\int_{B_2(0)\setminus B_1(0)}|y|^a u^q(y)dy \geq c$.
By this estimate and (\ref{2.10*}), we have
\begin{equation} \label{2.11*}
\begin{array}{ll}
u(x) &\geq c\displaystyle\int_{2|x|}^\infty (\frac{\int_{B_{t-|x|}(0)}|y|^{a-qa_0}dy}{
t^{n-\beta p}})^{\frac{1}{p-1}} \frac{dt}{t}\\[3mm]
&\geq c\displaystyle\int_{2|x|}^\infty (t^{\beta p+a-qa_0})^{\frac{1}{p-1}}
\frac{dt}{t}.
\end{array}
\end{equation}

When $\frac{q}{p-1} \in (0,\frac{\beta p+a}{n-\beta p}]$,
we have $\beta p+a-qa_0 \geq 0$. Eq. (\ref{2.11*}) implies $u(x)=\infty$. It is impossible.

Next, we consider the case $\frac{q}{p-1} \in (\frac{\beta p+a}{n-\beta p},
\frac{n+a}{n-\beta p})$.
Now (\ref{2.11*}) leads to
$$
u(x) \geq \frac{c}{|x|^{a_1}},
$$
where $a_1=\frac{q}{p-1}a_0-\frac{\beta p}{p-1}$.

{\it Substep 1.2.}
Write
\begin{equation} \label{2.14}
a_j=\frac{q}{p-1}a_{j-1}-\frac{\beta p+a}{p-1},~j=1,2,\cdots.
\end{equation}
Suppose that $a_k<a_{k-1}$ for $k=1,2,\cdots,j-1$.
By virtue of (\ref{2.10*}), it follows
$$\begin{array}{ll}
a_j-a_{j-1}&=(\displaystyle\frac{q}{p-1}-1)a_{j-1}-\frac{\beta p+a}{p-1}
<(\frac{q}{p-1}-1)a_0-\frac{\beta p+a}{p-1}\\[3mm]
&=(\displaystyle\frac{q}{p-1}-1)\frac{n-\beta p}{p-1}-\frac{\beta p+a}{p-1}
=(\frac{n-\beta p}{(p-1)^2})q-\frac{n+a}{p-1}\\[3mm]
&<(\displaystyle\frac{n-\beta p}{(p-1)^2})\frac{(n+a)(p-1)}{n-\beta p}-\frac{n+a}{p-1}
=0.
\end{array}
$$
Thus, $\{a_j\}_{j=0}^\infty$ is decreasing as long as (\ref{2.10*}) is true.

Furthermore, we claim that there must be $j_0>0$ such that $a_{j_0} \leq 0$.
Once it is true, similar to the argument in Substep 1.1, we also get
$\beta p+a-qa_{j_0-1} \geq 0$, which leads to $u(x)=\infty$.
This contradicts with the fact that $u$ is a positive solution.

In fact, by (\ref{2.14}) we get
$$
a_j=(\frac{q}{p-1})^ja_0-[1+\frac{q}{p-1}+\cdots+(\frac{q}{p-1})^{j-1}]
\frac{\beta p+a}{p-1}.
$$

If $\frac{q}{p-1}=1$, then we can find a large $j_0$ such that
$$
a_{j_0}=a_0-j_0\frac{\beta p+a}{p-1}
\leq 0.
$$

If $\frac{q}{p-1}\in (1,\frac{n+a}{n-\beta p})$, then
$a_0-\frac{\beta p+a}{q-p+1}<0$. We can find a large $j_0$ such that
$$\begin{array}{ll}
a_{j_0}&=(\displaystyle\frac{q}{p-1})^{j_0}a_0-\frac{(\frac{q}{p-1})^{j_0}-1}
{\frac{q}{p-1}-1}\frac{\beta p+a}{p-1}\\[3mm]
&=(\displaystyle\frac{q}{p-1})^{j_0}(a_0-\frac{\beta p+a}{q-p+1})
+\frac{\beta p+a}{q-p+1} \leq 0.
\end{array}
$$

If $\frac{q}{p-1} \in (0,1)$, letting $j \to \infty$, we get
$$
a_j=(\frac{q}{p-1})^ja_0-\frac{1-(\frac{q}{p-1})^j}
{1-\frac{q}{p-1}}\frac{\beta p+a}{p-1} \to \frac{\beta p+a}{q-p+1}<0.
$$
Thus, there must be $j_0$ such that
$a_{j_0} \leq 0$.

{\it Step 2.} Let $q=\frac{(n+a)(p-1)}{n-p\beta}$. We deduce the
contradiction if $u$ is a positive solution of (\ref{IE}).

For $R>0$, denote $B_R(0)$ by $B_R$. Using (\ref{IE})and the H\"older inequality,
we see that for any $x \in B_R$,
$$\begin{array}{ll}
u(x) &\geq c\displaystyle\int_0^R
(\frac{\int_{B_t(x)}|y|^au^q(y)dy}{t^{n-p\beta}})^{\frac{1}{p-1}}
\frac{dt}{t}\\[3mm]
&\geq cR^{-\frac{n-p\beta+1}{p-1}}
(\displaystyle\int_0^R(\int_{B_t(x)}|y|^au^q(y)dy)dt)^{\frac{1}{p-1}}.
\end{array}
$$
By exchanging the order of the integral variables, and noting
$B_{R/4} \times [R/4,R]$ is the subset of the cone $\{(y,t); t \in [|x-y|,
R], y \in B_R\}$, we have
$$\begin{array}{ll}
u(x) &\geq cR^{-\frac{n-p\beta+1}{p-1}}
(\displaystyle\int_{B_R}|y|^au^q(y)(\int_{|x-y|}^R dt)dy)^{\frac{1}{p-1}}\\[3mm]
&\geq cR^{-\frac{n-p\beta}{p-1}}
(\displaystyle\int_{B_{R/4}}|y|^au^q(y)dy)^{\frac{1}{p-1}}.
\end{array}
$$
Therefore, we get
\begin{equation} \label{low1}
|x|^a u^q(x) \geq c|x|^aR^{q\frac{p\beta-n}{p-1}}(\int_{B_{R/4}}|y|^a
u^q(y)dy)^{\frac{q}{p-1}}.
\end{equation}
Integrating on $B_{R/4}$ and using $q=\frac{(n+a)(p-1)}{n-p\beta}$, we obtain
$$\begin{array}{ll}
&\quad \displaystyle\int_{B_{R/4}}|x|^au^q(x)dx \\[3mm]
&\geq cR^{q\frac{p\beta-n}{p-1}}
\displaystyle\int_{B_{R/4}}|x|^adx(\int_{B_{R/4}}|y|^au^q(y)dy)^{\frac{q}{p-1}}\\[3mm]
&\geq c(\displaystyle\int_{B_{R/4}}|y|^au^q(y)dy)^{\frac{q}{p-1}}.
\end{array}
$$
Here $c$ is independent of $R$. Letting $R \to \infty$
and noting $q>p-1$, we have
\begin{equation} \label{tian}
\int_{R^n}|x|^a u^q(x) dx<\infty.
\end{equation}

Integrating (\ref{low1}) on $A_R=B_{R/4} \setminus B_{R/8}(0)$ yields
$$
\int_{A_R} |x|^a u^q(x)dx \geq c R^{q\frac{p\beta-n}{p-1}}\int_{A_R}|x|^adx
(\int_{B_{R/4}}|y|^a u^q(y)dy)^{\frac{q}{p-1}}.
$$
By $q=\frac{(n+a)(p-1)}{n-p\beta}$, it follows
$$
\int_{A_R} |x|^a u^q(x)dx \geq c(\int_{B_{R/4}}|y|^a
u^q(y)dy)^{\frac{q}{p-1}},
$$
where $c$ is independent of $R$. Letting $R \to \infty$, and noting
(\ref{tian}), we obtain
$$
\int_{R^n}|y|^a u^q(y)dy=0,
$$
which implies $u \equiv 0$. It is impossible.
\end{proof}

\subsection{Slow decay rate}

Theorem \ref{th2.1} shows that if (\ref{IE}) has the positive solution
$u$, then
\begin{equation} \label{cond}
q>\frac{(n+a)(p-1)}{n-p\beta}.
\end{equation}

To investigate the decay rates of $u$, we always assume (\ref{cond}) holds
hereafter.

By (\ref{cond}), we can see that $\frac{n-p\beta}{p-1}>\frac{p\beta+a}{q-p+1}$.
Thus, we call $\frac{n-p\beta}{p-1}$ the fast decay rate
and $\frac{p\beta+a}{q-p+1}$ the slow one.

Let $u$ be bounded but not integrable.
If it decays along $u(x) \simeq |x|^{-\theta}$
when $|x| \to \infty$, we prove that the rate $\theta$
must be the slow one $\frac{p\beta+a}{q-p+1}$.

\begin{theorem} \label{th5.1}
Let $r_0$ be an arbitrary given positive number in $[s_0,\infty)$.
Assume $u \in L^\infty(R^n) \setminus L^{r_0}(R^n)$ solves
(\ref{IE}) with (\ref{ASSUM}). If
$$
\lim_{|x| \to \infty}u(x)|x|^\theta \in (0,\infty),
$$
then $\theta$ must be the slow decay rate
$\frac{p\beta+a}{q-p+1}$.
\end{theorem}

\begin{proof}
{\it Step 1.}
Let $\theta<\frac{p\beta+a}{q-p+1}$. We claim that
there does not exist $C>0$ such that as $|x| \to \infty$,
$$
u(x) \geq C|x|^{-\theta}.
$$
This result shows that the decay rate of $u$ is not slower
than the slow one $\frac{p\beta+a}{q-p+1}$.

If there exists $C>0$ such that for some large $|x|$,
$$
u(x) \geq C|x|^{-\theta}, \quad \theta<\frac{p\beta+a}{q-p+1}.
$$
By an iteration we can deduce the contradiction.

Denote $\theta$ by $b_0$. Similar to the derivation of (\ref{2.11*}),
for $|x|>1$ we have
$$
u(x) \geq c|x|^{-b_1}, \quad b_1=\frac{qb_0-p\beta-a}{p-1}.
$$
By induction, for some large $|x|$ there holds
$$
u(x) \geq c|x|^{-b_j}, \quad
b_0=\theta, \quad
b_j=\frac{qb_{j-1}-p\beta-a}{p-1}, \quad j=1,2,\cdots.
$$
We claim that there must be $j_0$ such that $b_{j_0}<0$, which leads to
$u(x)=\infty$. In fact, similar to the proof of Theorem \ref{th2.1},
there also holds
$$
b_j=(b_0-\frac{p\beta+a}{q-p+1})(\frac{q}{p-1})^j+\frac{p\beta+a}{q-p+1}.
$$
Noting $q>p-1$ (which is implied by (\ref{cond})) and $b_0-\frac{p\beta+a}{q-p+1}<0$,
we can find a large $j_0$ such
that $b_{j_0}<0$. It is impossible since the solution $u$ blows up.

{\it Step 2.}
Let $\theta>\frac{p\beta+a}{q-p+1}$, $r_0 \geq \frac{n(q-p+1)}{p\beta+a}$.
If $u \not\in L^{r_0}(R^n)$, we claim that
there does not exist $C>0$ such that as $|x| \to \infty$,
$$
u(x) \leq C|x|^{-\theta}.
$$
This result shows that the decay rate of $u$ is not faster
than the slow rate $\frac{p\beta+a}{q-p+1}$.

Suppose there exists $C>0$ such that as $|x| \to \infty$,
$$
u(x) \leq C|x|^{-\theta}, \quad \textrm{where}
\quad \theta>\frac{p\beta+a}{q-p+1}.
$$
Since $u$ is bounded, for some large $R>0$, there holds
$$\begin{array}{ll}
\displaystyle\int_{R^n} u^{r_0}(x)dx=&\displaystyle\int_{B_R(0)}u^{r_0}(x)dx
+\int_{R^n \setminus B_R(0)} u^{r_0}(x)dx\\[3mm]
&\leq C+C\displaystyle\int_R^\infty r^{n-r_0\theta} \frac{dr}{r}.
\end{array}
$$
By virtue of $r_0 \geq \frac{n(q-p+1)}{p\beta+a}$, we see
$$
\int_{R^n} u^{r_0}(x)dx
<\infty,
$$
which contradicts with $u \not\in L^{r_0}(R^n)$.
\end{proof}

\paragraph{Remark 2.1.} Moreover, if the positive solution is radially symmetric and
decreasing about the origin, then as $|x| \to \infty$
$$
u(x)=O(|x|^{-\frac{p\beta+a}{q-p+1}}).
$$

In fact,
when $t \in (0,|x|/2)$, in $D:=B_t(x) \cap \{|y| \geq |x|\}$, there hold
$$
u(y) \geq u(x), \quad and \quad |y| \leq 3|x|/2.
$$
In addition, $|D| > \frac{1}{2}|B_t(x)| \geq ct^n$. Therefore,
$$
u(x) \geq c(|x|^a u^q(x))^{\frac{1}{p-1}}\int_0^{|x|/2}
(\frac{\int_D dy}{t^{n-p\beta}})^{\frac{1}{p-1}}\frac{dt}{t}
\geq c|x|^{\frac{p\beta+a}{p-1}} u^{\frac{q}{p-1}}(x).
$$
This implies $u(x) \leq C|x|^{-\frac{p\beta+a}{q-p+1}}$.

\paragraph{Remark 2.2.} Consider the positive solutions $u \not\in L^\infty(R^n)$.
We can find a singular solution with the slow rate $\frac{p\beta+a}{q-p+1}$.
Let
$$
u(x)=c|x|^{-t}.
$$

\begin{enumerate}
\item If (\ref{cond}) holds, we claim that $u(x)$ solves (\ref{PDE})
with
$$
t=\frac{p+a}{q-p+1}, \quad
c=t^{\frac{p-1}{q-p+1}}[n-1-(p-1)(t+1)]^{\frac{1}{q-p+1}}.
$$

In fact, if writing $u(x)=U(|x|)=U(r):=cr^{-t}$, we can see that
the left hand side of (\ref{PDE}) is
\begin{equation} \label{xifenr}
\begin{array}{ll}
&\quad -div(|\nabla u|^{p-2} \nabla u)\\[3mm]
&=-|U'|^{p-2}[(p-1)U''+\displaystyle\frac{n-1}{r}U']\\[3mm]
&=\displaystyle\frac{c^{p-1}t^{p-1}}{r^{(p-2)(t+1)+(t+2)}}
[n-1-(p-1)(t+1)].
\end{array}
\end{equation}
By virtue of (\ref{cond}), we get $n-1>(p-1)(t+1)$, and hence
the value of the result above is positive. Noting the values of $t$ and $c$,
we see that (\ref{xifenr}) is equal to $c^qr^{a-tq}$, which is the exact
right hand side of (\ref{PDE}). In addition, the asymptotic rate $t$ is the
slow one when $|x| \to \infty$ and $|x| \to 0$, respectively.

\item We can find a double bounded function $R(x)$ such that $u(x)=c|x|^{-t}$
also solves (\ref{IE}). Here $t=\frac{p\beta}{q-p+1}$.

In fact, write
$$
W_1=\int_0^{|x|/2}[\frac{\int_{B_t(x)} |y|^a
u^q(y)dy}{t^{n-p\beta}}] ^{\frac{1}{p-1}}\frac{dt}{t},
$$
$$
W_2=\int_{|x|/2}^\infty [\frac{\int_{B_t(x)}|y|^a
u^q(y)dy}{t^{n-p\beta}}] ^{\frac{1}{p-1}}\frac{dt}{t}.
$$

If $t \in (0,|x|/2)$, we have $|x|/2<|y|<3|x|/2$ for $y \in B_t(x)$. Thus,
$$
c_1|x|^{a-qt}t^n \leq \int_{B_t(x)}|y|^au^q(y)dy \leq c_2|x|^{a-qt}t^n,
$$
where $0<c_1 \leq c_2$, and hence
$$
c_1|x|^{\frac{p\beta+a-qt}{p-1}} \leq W_1 \leq c_2|x|^{\frac{p\beta+a-qt}{p-1}}.
$$

If $t \geq |x|/2$, we have $B_t(x) \subset B(0,|x|+t) \subset B(0,3t)$. Thus,
by (\ref{cond}),
$$
\int_{B_t(x)}u^q(y)dy \leq C\int_{B_{3t}(0)}|y|^{a-qt}dy
\leq Ct^{n+a-qt}.
$$
Therefore,
$$
W_2 \leq C|x|^{\frac{p\beta+a-qt}{p-1}}.
$$
Noticing $t=\frac{p\beta+a}{q-p+1}$, and combining the estimates of $W_1$
and $W_2$, we obtain
$$
c_1(W_1+W_2) \leq u(x) \leq c_2(W_1+W_2).
$$
Setting
$$
R(x)=u(x)[W_1+W_2]^{-1},
$$
we know that $R(x)$ is double
bounded and $u(x)$ solves (\ref{IE}).
\end{enumerate}

\section{Integability and fast decay rate}

\subsection{Integrability}

\begin{theorem} \label{inter}
Assume $u \in L^{s_0}(R^n)$ solves (\ref{IE}) with (\ref{ASSUM}), where $s_0=\frac{n(q-p+1)}{p\beta+a}$. Then
\begin{equation}
u \in L^s(R^n), \quad \forall s > \frac{n(p-1)}{n-\beta p}.
\label{intervalSW}
\end{equation}
In addition, the lower bound $\frac{n-\beta p}{n(p-1)}$ of $s$ is optimal.
\end{theorem}

\begin{proof}
{\it Step 1.}
For $A>0$, set
$$\begin{array}{lll}
&u_A(x)=w(x), &\quad if \quad u(x)>A \quad or \quad |x|>A;\\
&u_A(x)=0, &\quad otherwise,
\end{array}
$$
and $u_B(x)=u(x)-u_A(x)$. Let $\sigma$ satisfy
\begin{equation} \label{2.7}
\frac{2-p}{s_0}<\frac{1}{\sigma}<\frac{2-p}{s_0}
+\frac{n-\beta p}{n}.
\end{equation}
For $g \in L^\sigma(R^n)$,
define operators $T$ and $S$,
$$
Tg(x):=R(x)\int_0^\infty (\frac{\int_{B_t(x)}|y|^au^{q}(y)dy}{t^{n-\beta p}})^{\frac{2-p}{p-1}}
\frac{\int_{B_t(x)}|y|^au_A^{q-1}(y)g(y)dy}{t^{n-\beta p}} \frac{dt}{t}
$$
$$
Sg(x):=\int_0^\infty (\frac{\int_{B_t(x)}|y|^au_A^{q-1}(y)g(y)dy}{t^{n-\beta p}})^{\frac{1}{p-1}}
\frac{dt}{t}
$$
and write
$$
F(x):=R(x)\int_0^\infty (\frac{\int_{B_t(x)}|y|^au^{q}(y)dy}{t^{n-\beta p}})^{\frac{2-p}{p-1}}
\frac{\int_{B_t(x)}|y|^au_B^{q}(y)dy}{t^{n-\beta p}} \frac{dt}{t}
$$
Clearly, $u$ is a solution of the following equation
$$
g=Tg+F.
$$

{\it Step 2.}
$T$ is a contraction map from $L^\sigma(R^n)$ into itself.

In fact, by the H\"older inequality, there holds $|Tg| \leq
Cu^{2-p} |Sg|^{p-1}.$ Therefore, we get
\begin{equation} \label{2.3}
\|Tg\|_\sigma \leq C\|u\|_{s_0}^{2-p} \|Sg\|_{\tau}^{p-1}
\end{equation}
where $\tau>0$ satisfies
\begin{equation} \label{2.4}
\frac{1}{\sigma}=\frac{2-p}{s_0}+\frac{p-1}{\tau}.
\end{equation}
By (\ref{2.7}) and (\ref{2.4}), we get
\begin{equation} \label{2.2}
0<\frac{p-1}{\tau}<1-\frac{\beta p}{n}.
\end{equation}
Therefore, we can use the weighted Hardy-Littlewood-Sobolev inequality
and the Wolff type inequality to obtain
\begin{equation} \label{2.5}
\|Sg\|_{\tau} \leq C\|u_A^{q-1}
g\|_{\frac{n\tau}{n(p-1)+\tau(\beta p+a)}}^{\frac{1}{p-1}}.
\end{equation}
Since (\ref{2.4}) and $s_0=\frac{n(q-p+1)}{p\beta+a}$ lead to
$\frac{p-1}{\tau}-\frac{1}{\sigma} =\frac{q-1}{s_0}-\frac{\beta
p}{n},$ it follows from (\ref{2.5}) and the H\"older inequality
that $\|Sg\|_\tau^{p-1} \leq C\|u_A\|_{s_0}^{q-1} \|g\|_\sigma.$
Inserting this into (\ref{2.3}) yields
\begin{equation} \label{2.6}
\|Tg\|_\sigma \leq C\|u\|_{s_0}^{2-p}
\|u_A\|_{s_0}^{q-1}
\|g\|_\sigma.
\end{equation}

By virtue of $u \in L^{s_0}(R^n)$, $C\|u\|_{s_0}^{2-p}
\|u_A\|_{s_0}^{q-1} \leq \frac{1}{2}$ when $A$ is sufficiently
large. Then $T$ is a shrinking operator. Noticing that $T$ is
linear, we know that $T$ is a contraction map from
$L^{\sigma}(R^n)$ to itself as long as $\sigma$ satisfies
(\ref{2.7}).

{\it Step 3.} Estimating $F$ to lift the regularity.

Similar to (\ref{2.3}) and (\ref{2.5}), for all $\sigma$ satisfying (\ref{2.7}),
there holds
$$
\|F\|_\sigma \leq C\|u\|_{s_0}^{2-p}
\|u_B^{q}\|_{\frac{n\tau}{n(p-1)+\tau(\beta p+a)}}
$$
where $\tau$ satisfies (\ref{2.2}). Noting $u\in L^{s_0}(R^n)$ and
the definition of $u_B$, we see that $F \in L^{\sigma}(R^n)$ as
long as $\sigma$ satisfies (\ref{2.7}). Taking $X=L^{s_0}(R^n)$,
$Y=L^{\sigma}(R^n)$ and $Z=L^{s_0}(R^n) \cap L^{\sigma}(R^n)$ in
Lemma 2.1 of \cite{JL}, we have $u \in L^{\sigma}(R^n)$ for all
$\sigma$ satisfying (\ref{2.7}).

{\it Step 4.}
Extend the interval from (\ref{2.7}).

Let
\begin{equation} \label{2.10}
\frac{1}{s} \in (0,\frac{n-\beta p}{n(p-1)}).
\end{equation}
Thus, we can use the weighted Hardy-Littlewood-Sobolev inequality
and the Wolff type inequality to deduce that
\begin{equation}
\|u\|_s \leq C\|u^{q}\|_{\frac{ns}{n(p-1)+s(\beta p+a)}}^{\frac{1}{p-1}}
\leq C\|u\|_{\frac{nsq}{n(p-1)+s(\beta p+a)}}^{\frac{q}{p-1}}.
\label{2.8}
\end{equation}
Noting (\ref{2.7}), from (\ref{2.8}) we see that $\|u\|_s <\infty$
as long as $s$ satisfies
\begin{equation} \label{2.9}
\frac{2-p}{s_0}< \frac{n(p-1)+s(\beta p+a)}{nsq}
<\frac{2-p}{s_0}+\frac{n-\beta p}{n}.
\end{equation}

Next, we will prove that
\begin{equation} \label{claim}
\textrm{Eq. (\ref{2.9}) is true as long as
(\ref{2.10}) holds}.
\end{equation}

First, $(p-1)(q-1)>0$ leads to $q(2-p)<q-(p-1)$. Thus,
$\frac{2-p}{q-p+1}<\frac{1}{q}$. Multiplying by $\frac{p\beta+a}{n}$
yields
\begin{equation} \label{2.11}
\frac{2-p}{s_0}< \frac{\beta p+a}{nq}.
\end{equation}
Second, (\ref{cond}) shows $n+a \geq \frac{q(p\beta+a)}{q-(p-1)}$.
Hence, $n-p\beta \geq \frac{(p-1)(p\beta+a)}{q-(p-1)}$.
Multiplying by $(q-1)$ yields $(n-p\beta)(q-1) \geq
(p\beta+a)(1-\frac{q(2-p)}{q-p+1})$ or
$$
q\frac{(2-p)(p\beta+a)}{q-p+1}+q(n-p\beta)-(p\beta+a) \geq n-p\beta.
$$
Multiplying by $\frac{1}{n(p-1)}$, we get
\begin{equation} \label{2.12}
[\frac{2-p}{s_0}+\frac{n-p\beta}{n}-\frac{p\beta+a}{nq}]\frac{q}{p-1} \geq \frac{n-p\beta}{n(p-1)}.
\end{equation}

By using (\ref{2.11}) and (\ref{2.12}), we can see (\ref{claim}).

{\it Step 5.} We claim that $\frac{n-\beta p}{n(p-1)}$ is optimal. In fact,
for sufficiently large $|x|$, from (\ref{IE}) we deduce that
\begin{equation} \label{low}
\begin{array}{ll}
u(x) &\geq c \displaystyle\int_{2|x|}^{4|x|}
(\frac{\int_{B_2(0)\setminus B_1(0)}|y|^au^q(y)dy}{t^{n-\beta p}})^{\frac{1}{p-1}}
\frac{dt}{t}\\[3mm]
&\geq c\displaystyle\int_{2|x|}^{4|x|}(\frac{1}{t^{n-\beta p}})^{\frac{1}{p-1}}
\frac{dt}{t} \geq c|x|^{\frac{\beta p-n}{p-1}}.
\end{array}
\end{equation}
If $\frac{1}{s} \geq \frac{n-\beta p}{n(p-1)}$, then for some large constant $d>0$,
$$
\|u\|_{L^s(R^n\setminus B_d(0))}^s \geq c\int_d^{\infty}r^{n-s\frac{n-\beta p}{p-1}}
\frac{dr}{r} =\infty.
$$
Theorem \ref{inter} is proved.
\end{proof}

\begin{theorem} \label{boundth}
Assume $u \in L^{s_0}(R^n)$ solve (\ref{IE}) with (\ref{ASSUM}). Then $u$ is bounded in $R^n$.
\end{theorem}

\begin{proof}
In view of (\ref{IE}),
$$\begin{array}{ll}
u(x)
&\leq C(\displaystyle\int_0^1[\frac{\int_{B_t(x)}|y|^au^q(y)
dy}{t^{n-\beta p}}]^{\frac{1}{p-1}} \frac{dt}{t}\\[3mm]
&+\displaystyle\int_1^{\infty}[\frac{\int_{B_t(x)}|y|^au^q(y)
dy}{t^{n-\beta p}}]^{\frac{1}{p-1}} \frac{dt}{t})\\[3mm]
&:=C(H_1+H_2).
\end{array}
$$

By H\"older's inequality, for any $l>1$ satisfying $n+\frac{al}{l-1}>0$, we have
\begin{equation} \label{frac}
\int_{B_t(x)}|y|^au^q(y) dy \leq C\|u^{q}\|_l(\int_{B_t(x)}|y|^{\frac{al}{l-1}}dy)^{1-1/l}.
\end{equation}
When $t \geq |x|/2$,
$$
\int_{B_t(x)}|y|^{\frac{al}{l-1}}dy \leq \int_{B_{|x|+t}(0)}|y|^{\frac{al}{l-1}}dy
\leq Ct^{n+\frac{al}{l-1}}.
$$
When $t<|x|/2$, $|y|>|y-x|$ for all $y \in B_t(x)$. Hence
$$
\int_{B_t(x)}|y|^{\frac{al}{l-1}}dy \leq \int_{B_{t}(x)}|y-x|^{\frac{al}{l-1}}dy
\leq Ct^{n+\frac{al}{l-1}}.
$$
Substituting these estimates into (\ref{frac}), we get
$$
\int_{B_t(x)}|y|^au^q(y) dy \leq C\|u^{q}\|_l t^{(1-1/l)n+a}.
$$
Take $l$ sufficiently large such that $ql>\frac{n(p-1)}{n-p\beta}$
and $p\beta+a-n/l>0$.
According to Theorem \ref{inter}, $\|u^{q}\|_l<\infty$.
Therefore,
$$
H_1 \leq C\int_0^1(\frac{t^{(1-1/l)n+a}}{t^{n
-\beta p}})^{\frac{1}{p-1}} \frac{dt}{t}
\leq C\int_0^1t^{\frac{\beta p+a-n/l}{p-1}} \frac{dt}{t}
\leq C.
$$

If $z \in B_\delta(x)$, then $B_t(x) \subset B_{t+\delta}(z)$.
For $\delta \in (0,1)$ and $z \in B_\delta(x)$,
\begin{equation} \label{tec}
\begin{array}{ll}
H_2&=\displaystyle\int_1^{\infty}[\frac{\int_{B_t(x)}|y|^au^q(y)
dy}{t^{n-\beta p}}]^{\frac{1}{p-1}} \frac{dt}{t}\\[3mm]
&\leq
\displaystyle\int_1^{\infty} (\frac{\int_{B_{t+\delta}(z)}|y|^au^q(y)
dy}{(t+\delta)^{n-\beta p}})^{\frac{1}{p-1}}
(\frac{t+\delta}{t})^{\frac{n-\beta p}{p-1}+1}
\frac{d(t+\delta)}{t+\delta}\\[3mm]
&\leq (1+\delta)^{\frac{n-\beta p}{p-1}+1}
\displaystyle\int_{1+\delta}^{\infty}
(\frac{\int_{B_t(z)}|y|^au^q(y)dy}{t^{n-\beta p}})^{\frac{1}{p-1}}
\frac{dt}{t}\leq Cu(z).
\end{array}
\end{equation}

Combining the estimates of $H_1$ and $H_2$, we have $u(x) \leq
C+Cu(z)$ for $z \in B_\delta(x),$ where $\delta \in (0,1)$.
Integrating on $B_\delta(x)$, we get
$$\begin{array}{ll}
|B_\delta(x)|u(x) &\leq C+C\displaystyle\int_{B_\delta(x)}u(z)dz \\[3mm]
&\leq
C+C\|u\|_{s_0}
|B_\delta(x)|^{1-\frac{1}{s_0}} \leq C.
\end{array}
$$
This shows $u$ is bounded in $R^n$. Theorem \ref{boundth} is proved.
\end{proof}

\subsection{Fast decay rate}

\begin{theorem} \label{decayth}
Assume $u \in L^{s_0}(R^n)$ solves (\ref{IE}) with (\ref{ASSUM}). Then
\begin{equation} \label{decay}
\lim_{|x| \to \infty}u(x)=0.
\end{equation}
\end{theorem}

\begin{proof}
Take $x_0 \in R^n$. By Theorem \ref{boundth},
$\|u\|_\infty<\infty$. Thus, $\forall \varepsilon>0$, there
exists $\delta \in (0,1)$ such that
$$
\int_0^{\delta}
[\frac{\int_{B_t(x_0)}|z|^a u^q(z)dz}{t^{n-\beta p}}]^{\frac{1}{p-1}}
\frac{dt}{t} \leq C\|u\|_\infty^{\frac{q}{p-1}}
\int_0^{\delta} t^{\frac{\beta p+a}{p-1}} \frac{dt}{t}
<\varepsilon.
$$
On the other hand, similar to the derivation of (\ref{tec}), as
$|x-x_0|<\delta$,
$$
\int_{\delta}^{\infty}
[\frac{\int_{B_t(x_0)}|z|^au^q(z)dz}{t^{n-\beta p}}]^{\frac{1}{p-1}}
\frac{dt}{t} \leq Cu(x).
$$
Combining these estimates, we get
$$
u(x_0)<\varepsilon+Cu(x), \quad for \quad |x-x_0|
<\delta.
$$
Since $u \in
L^{s_0}(R^n)$, there holds $\lim_{|x_0| \to
\infty}\int_{B_\delta(x_0)}u^{s_0}(x)dx=0$. Thus, we have
\begin{equation}
\begin{array}{ll}
u^{s_0}(x_0)
&=|B_\delta(x_0)|^{-1}\displaystyle\int_{B_\delta(x_0)}
u^{s_0}(x_0)dx\\[3mm] &\leq
C\varepsilon^{s_0}
+C|B_\delta(x_0)|^{-1}\displaystyle\int_{B_\delta(x_0)}u^{s_0}(x)dx
\to 0
\end{array}
\label{3.15}
\end{equation}
when $|x_0| \to \infty$ and $\varepsilon \to 0$.
Thus, (\ref{decay}) is proved.
\end{proof}

\begin{theorem} \label{rateth1}
Assume $u \in L^{s_0}(R^n)$ solves (\ref{IE}) with (\ref{ASSUM}). Then we can find $c>0$ such that
$u(x) \geq c |x|^{\frac{\beta p-n}{p-1}}$ when $|x| \to \infty$.
\end{theorem}

\begin{proof}
Clearly,
$\int_{B_2(0)\setminus B_1(0)}u^p(y)v^q(y)dy \geq c>0$. It
follows that
$$\begin{array}{ll}
u(x)
&\geq c\displaystyle\int_{|x|+2}^{\infty}
[\frac{\int_{B_2(0)\setminus B_1(0)}|y|^au^q(y)dy}{t^{n-\beta p}}]^{\frac{1}{p-1}}
\frac{dt}{t}\\[3mm]
&\geq c\displaystyle\int_{|x|+2}^{\infty}
t^{-\frac{n-\beta p}{p-1}}
\frac{dt}{t} \geq c|x|^{-\frac{n-\beta p}{p-1}}.
\end{array}
$$
Theorem \ref{rateth1} is proved.
\end{proof}

\begin{theorem} \label{rateth2}
Assume $u \in L^{s_0}(R^n)$ solves (\ref{IE}) with (\ref{ASSUM}). Then we can find $C>0$ such that
$u(x) \leq C|x|^{\frac{\beta p-n}{p-1}}$ when $|x| \to \infty$.
\end{theorem}

\begin{proof}
Take a cutting-off function
$\psi(x) \in C_0^{\infty}(B_{2} \setminus B_1)$
satisfying
$$
0 \leq \psi(x) \leq 1, \quad for \quad 1 \leq |x| \leq 2;
$$
$$
\psi(x)=1, \quad for \quad \frac{5}{4} \leq |x| \leq \frac{7}{4}.
$$
For any $\rho>0$, set $\psi_\rho(x)=\psi(\frac{x}{\rho})$. Define
$$
h(x)=u(x)|x|^{(n+a)/q}\psi_\rho(x).
$$
Then, one of the following two cases holds:

(1) There exists a positive constant
$C$ (independent of $\rho$) such that
\begin{equation}
h(x) \leq C, \quad \forall x; \label{3.3*}
\end{equation}

(2) There exists an increasing
sequence $\{\rho_j\}_{j=1}^{\infty}$ satisfying
$
\lim_{j \to \infty}\rho_j=\infty,
$
such that as $x_{\rho_j} \in B_{2\rho_j}\setminus B_{\rho_j}$,
\begin{equation}
\lim_{j \to \infty}h(x_{\rho_j})=\infty. \label{3.4}
\end{equation}

{\it Step 1.} If (\ref{3.3*}) is true, then for large $|x|$,
\begin{equation}
u(x) \leq C|x|^{-(n+a)/q}. \label{3.5}
\end{equation}
When $t \in (0,|x|/2)$, $y \in B_t(x)$ implies $|x|/2 \leq |y|
\leq 3|x|/2$, which leads to $u^q(y) \leq C|x|^{-(n+a)}$. In addition,
$|y| \geq |x|/2 \geq |y-x|$ leads to
$$
\int_{B_t(x)}|y|^ady \leq \int_{B_t(x)}|y-x|^ady \leq Ct^{n+a}.
$$
Thus, we have
\begin{equation} \label{3.6}
\int_0^{\frac{|x|}{2}}
(\frac{\int_{B_t(x)}|y|^au^{q}(y)dy}{t^{n-\beta
p}})^{\frac{1}{p-1}} \frac{dt}{t}\leq
\displaystyle\frac{C}{|x|^{\frac{n+a}{p-1}}}
\int_0^{\frac{|x|}{2}} t^{\frac{\beta p+a}{p-1}} \frac{dt}{t} \leq
\frac{C}{|x|^{\frac{n-\beta p}{p-1}}}.
\end{equation}

On the other hand, Theorem \ref{boundth} and $n+a>0$ imply
$$
\int_{B_1(0)}|y|^au^q(y)dy \leq C\|u\|_\infty^q \int_{B_1(0)}|y|^ady<\infty.
$$
Noting (\ref{cond}), by the H\"older inequality and Theorem \ref{inter}, we get
$$
\int_{R^n \setminus B_1(0)}|y|^au^q(y)dy \leq C\|u^q\|_{k'}
(\int_{R^n\setminus B_1(0)}|y|^{ak}dy)^{1/k}<\infty,
$$
where $\frac{1}{k}=\frac{-\epsilon-a}{n}$ and $\frac{1}{k'}=1-\frac{1}{k}$
with $\epsilon>0$ sufficiently small. Combining two estimates above yields
\begin{equation} \label{time9}
\int_{R^n}|y|^au^q(y)dy<\infty.
\end{equation}
Then
$$
\int_{|x|/2}^{\infty}
(\frac{\int_{B_t(x)}|y|^au^{q}(y)dy}{t^{n-\beta p}})^{\frac{1}{p-1}}
\frac{dt}{t} \leq C\int_{|x|/2}^{\infty}
t^{\frac{\beta p-n}{p-1}} \frac{dt}{t} \leq
C|x|^{-\frac{n-\beta p}{p-1}}.
$$
Combining this result with (\ref{3.6}), we
obtain
$$
u(x)=R(x)W_{\beta,p}(u^{q})(x) \leq
C|x|^{-\frac{n-\beta p}{p-1}}.
$$
Theorem \ref{rateth2} is proved in the case of (1).

{\it Step 2.} We prove case (2) does not happen.

Let $x_\rho$ be the maximum point of $h(x)$ in
$B_{2\rho}\setminus B_{\rho}$. It follows from (\ref{3.4}) that
\begin{equation}
u(x_{\rho_j})=\frac{h(x_{\rho_j})}{\psi_{\rho_j}(x_{\rho_j})|x_{\rho_j}|^{n/q}}
\geq \frac{c}{\rho_j^{n/q}}. \label{3.7}
\end{equation}
For convenience, we denote $\rho_j$ by $\rho$.

We also obtain that $\psi_\rho(x_\rho)>\delta$
for some $\delta>0$ (independent of $\rho$).
The details of the proof can be seen in \cite{LeiLi}.
Therefore, by the smoothness of $\psi$, we can find a suitably small positive
constant $\sigma \in (0,1/2)$, such that $ \psi_\rho(y)>\delta/2 $ for
$|y-x_\rho|<\sigma|x_\rho|$. Hence, by $h(y) \leq h(x_\rho)$, we get
\begin{equation}
u(y) \leq C\frac{u(x_\rho)}{\psi_\rho(y)} \leq
C(\delta) u(x_\rho), \quad as \quad |y-x_\rho|<\sigma|x_\rho|.
\label{3.10}
\end{equation}

Clearly,
\begin{equation}
\begin{array}{ll}
u(x_\rho) &\leq C[\displaystyle\int_0^{\sigma|x_\rho|}
(\frac{\int_{B_t(x_\rho)}|y|^au^{q}(y)dy}{t^{n-\beta p}})^{\frac{1}{p-1}}
\frac{dt}{t}\\[3mm]
&\quad +\displaystyle\int_{\sigma|x_\rho|}^{\infty}
(\frac{\int_{B_t(x_\rho)}|y|^au^{q}(y)dy}{t^{n-\beta p}})^{\frac{1}{p-1}}
\frac{dt}{t}]\\[3mm]
&:=C(J_1+J_2).
\end{array}
\label{3.11}
\end{equation}
From (\ref{time9}), it follows
\begin{equation}
J_2 \leq C\int_{\sigma|x_\rho|}^{\infty}
t^{-\frac{n-\beta p}{p-1}} \frac{dt}{t} \leq
C|x_\rho|^{-\frac{n-\beta p}{p-1}}. \label{3.12}
\end{equation}
Using (\ref{3.10}), we obtain that, for $r \in
(0,\sigma|x_\rho|)$,
\begin{equation}
\begin{array}{ll}
J_1 &\leq Cu(x_\rho)[\displaystyle\int_0^r
(\frac{\int_{B_t(x_\rho)}|y|^au^{q-p+1}(y)dy}{t^{n-\beta p}})^{\frac{1}{p-1}}
\frac{dt}{t}\\[3mm]
&\quad +\displaystyle\int_r^{\sigma|x_\rho|}
(\frac{\int_{B_t(x_\rho)}|y|^au^{q-p+1}(y)dy}{t^{n-\beta p}})^{\frac{1}{p-1}}
\frac{dt}{t}]\\[3mm]
&:=Cu(x_\rho)(J_{11}+J_{12}).
\end{array}
\label{3.13}
\end{equation}
In view of $\sigma \in (0,1/2)$, $|y|^a \leq C|x_\rho|^a$.
According to Theorem \ref{decayth}, for any $\varepsilon \in (0,1)$, there holds
$$
J_{11} \leq C\|w\|_{L^\infty(B_{\sigma|x_\rho|}(x_\rho))}^{\frac{q-p+1}{p-1}}
|x_\rho|^{\frac{a}{p-1}}\int_0^r t^{\frac{\beta p}{p-1}} \frac{dt}{t}
\leq C\varepsilon |x_\rho|^{\frac{a}{p-1}}
$$
as long as $\rho$ is sufficiently large. On the other hand, by
H\"older's inequality and Theorem 3.1,
$$
\int_{B_t(x_\rho)}|y|^au^{q-p+1}(y)dy \leq
\|u^{q-p+1}\|_{k'} (\int_{B_t(x_\rho)}|y|^{ak} dy)^{\frac{1}{k}}
\leq Ct^{n/k+a},
$$
where $\frac{1}{(q-p+1)k'}=\frac{n-\beta p-\epsilon}{n(p-1)}$
with $\epsilon>0$ sufficiently small.
Hence,
$$
J_{12} \leq C\int_r^{\sigma|x_\rho|}
t^{[\beta p-\frac{q-p+1}{p-1}(n-p\beta-\epsilon)+a]/(p-1)} \frac{dt}{t}.
$$
By virtue of (\ref{cond}), $\beta p-\frac{q-p+1}{p-1}(n-p\beta-\epsilon)+a<0$
as long as $\epsilon$ is sufficiently small. Therefore,
if $\rho$ is sufficiently large and $r$ is chosen suitably large,
then
$$
J_{12} \leq \varepsilon.
$$

Substituting the estimates of $J_{11}$ and $J_{12}$ into (\ref{3.13}),
we obtain
$$
J_1 \leq C\varepsilon u(x_\rho)
$$
when $\rho$ is sufficiently large. Inserting this result and
(\ref{3.12}) into (\ref{3.11}), and choosing $\varepsilon$ sufficiently
small, we get
$$
u(x_\rho) \leq C|x_\rho|^{-\frac{n-\beta p}{p-1}}.
$$
By (\ref{3.10}), we obtain that as $|x-x_\rho|<\sigma|x_\rho|$,
$$
u(x) \leq Cu(x_\rho) \leq
C|x_\rho|^{-\frac{n-\beta p}{p-1}} \leq
C|x|^{-\frac{n-\beta p}{p-1}}.
$$
Since $\rho$ is arbitrary, the result above still holds for all
$x$ as long as $|x|$ is large. This result contradicts (\ref{3.4})
if we notice (\ref{cond}).
Thus, case (2) does not happen.
\end{proof}

\paragraph{Proof of Theorem \ref{th1.2}.}

By the argument in section 2, we only need to prove item 1 of Theorem \ref{th1.2}.

By Theorems \ref{boundth}, \ref{rateth1} and \ref{rateth2}, we see that if
$u \in L^{s_0}(R^n)$, then $u$ is bounded and decays with the fast rate.
On the contrary, if $u$ is bounded and decays with the fast rate, we have
$$
\int_{R^n}u^{s_0}dx =\int_{B_R(0)}u^{s_0}dx +\int_{R^n\setminus B_R(0)}u^{s_0}dx
\leq C+C\int_R^\infty r^{n-\frac{n-p\beta}{p-1}\frac{n(q-p+1)}{p\beta+a}}\frac{dr}{r}.
$$
Noting (\ref{cond}), we see that $u \in L^{s_0}(R^n)$.

\section{Results on PDE}

\subsection{Integral equation}

\begin{theorem} \label{equiva}
Let $u$ be a positive weak solution of (\ref{1}) satisfying
$\inf_{R^n}u=0$. Then there exists a positive function $R(x)$ such
that
\begin{equation}
u(x)=R(x)W_{1,p}(|y|^au^q(y))(x), \quad in \quad R^n. \label{3.1}
\end{equation}
Moreover, there exist positive constants $C_i, (i=1,2)$ such that
\begin{equation}
C_1 \leq R(x) \leq C_2. \label{3.2}
\end{equation}
\end{theorem}

\begin{proof}
Since $u$ is a positive weak solution of (\ref{1}), it is a
$\mathcal{A}$-superharmonic function. According to Corollary 4.13
in \cite{KM}, by $\inf_{R^n}u=0$ we can find two positive
constants $C_1$ and $C_2$ such that
\begin{equation}
C_1W_{1,p}(u^p)(x) \leq u(x) \leq C_2W_{1,p}(u^p)(x),
\quad x \in R^n. \label{3.3}
\end{equation}
Set
$$
R(x)=\frac{u(x)}{W_{1,p}(u^p)(x)}.
$$
Then the solution $u$ of (\ref{PDE}) satisfies (\ref{3.1}). At
the same time, (\ref{3.3}) leads to (\ref{3.2}).
\end{proof}

As a corollary of Theorems \ref{th1.2} and \ref{equiva},
we can see the following results.

\begin{theorem} \label{tiant}
If $q \in (0,\frac{(n+a)(p-1)}{n-p}]$, (\ref{1}) has no positive
solution. Let $u$ be a positive weak solution of (\ref{1}) with
(\ref{ASSUM}). Then

(1) $u$ is bounded and decays with the fast rate $\frac{n-p}{p-1}$
if and only if $u \in L^{s_0}(R^n)$ with $\beta=1$.

(2) If $u \in L^\infty(R^n) \setminus L^{s_0}(R^n)$ decays with some rate, then the
rate must be the slow one $\frac{p+a}{q-p+1}$.
\end{theorem}

In particular, for the positive solutions of $p$-Laplace equation
(\ref{PDE}), we also have the same conclusions.

\subsection{Finite energy solutions and integrable solutions}

The following theorem shows that the integrable solutions
satisfy the invariant property of the system and the norm under
the scaling $u_\lambda(x)=\lambda^\theta u(\lambda x)$ with
$\lambda>0$.

\begin{theorem} \label{prop6.1}
Assume the scaling $u_\lambda(x)$ is still a solution of
(\ref{IE}) with some new double bounded function $R(x)$, if and
only if $\theta=\frac{p\beta+a}{q-p+1}$. Moreover,
$\|u_\lambda\|_\eta=\|u\|_\eta$ if and only if $\eta=s_0$.
\end{theorem}

\paragraph{Remark 4.1.} Let $a=0$. If
$q$ is equal to the critical exponent $p^*-1$, then $\eta=q+1$.

\begin{proof}
Clearly,
$$\begin{array}{ll}
u_\lambda(x)&=R(\lambda x)\lambda^\theta
\displaystyle\int_0^\infty(\frac{\int_{B_t(\lambda
x)}|y|^au^q(y)dy}
{t^{n-p\beta}})^{\frac{1}{p-1}}\frac{dt}{t}\\[3mm]
&=R(\lambda x)\lambda^\theta
\displaystyle\int_0^\infty(\frac{\int_{B_{t/\lambda}(\lambda x)}
|\lambda y|^au^q(\lambda y)\lambda^ndy}
{t^{n-p\beta}})^{\frac{1}{p-1}}\frac{dt}{t}\\[3mm]
&=R(\lambda x)\lambda^{\theta+\frac{p\beta+a-q\theta}{p-1}}
\displaystyle\int_0^\infty(\frac{\int_{B_{s}(x)}|z|^au_\lambda^q(z)dz}
{s^{n-p\beta}})^{\frac{1}{p-1}}\frac{ds}{s}.
\end{array}
$$
Here $R(\lambda x)$ is a new double bounded function.
$u_\lambda(x)$ is still a weak solution of (\ref{IE}), if and only
if the exponent of $\lambda$ is equal to zero. Clearly,
$\theta+\frac{p\beta+a-q\theta}{p-1}=0$ implies
$\theta=\frac{p\beta+a}{q-p+1}$.

In addition,
$$
\int_{R^n}u_\lambda^\eta(x)dx
=\lambda^{\eta\theta-n}\int_{R^n}u^\eta(z)dz.
$$
The energy is invariant if and only if $\eta=\frac{n}{\theta}$.
Inserting the value of $\theta$ we obtain
$\eta=\frac{n(q-p+1)}{p\beta+a}$.
\end{proof}

According to Theorem \ref{equiva}, the positive weak solution $u$
of (\ref{1}) also solves (\ref{IE}). Hence, the result above holds
for (\ref{1}). In particular, the corresponding result is still
true for (\ref{PDE}). In fact, we can see it by a formal
calculation. If denoting $y=\lambda x$, then we have
$$\begin{array}{ll}
&\quad |x|^au_\lambda^q(x)=\lambda^{q\theta-a}|y|^au(y)\\[3mm]
&=-\lambda^{q\theta-a}div_y[|\nabla_yu(y)|^{p-2}\nabla_yu(y)]\\[3mm]
&=-\lambda^{q\theta-a-(\theta+1)(p-1)-1}
div_x[|\nabla_xu(x)|^{p-2}\nabla_xu(x)].
\end{array}
$$
Therefore, $u_\lambda$ solves (\ref{PDE}) if and only if
$q\theta-a-(\theta+1)(p-1)-1=0$, which implies
$\theta=\frac{p+a}{q-p+1}$. On the other hand, $\|u_\lambda\|_\eta
=\|u\|_\eta$ if and only if
$\eta=\frac{n}{\theta}=\frac{n(q-p+1)}{p+a}$.

\vskip 5mm

By the following theorem, we can introduce another
solution--the finite energy solution.

\begin{theorem} \label{prop6.2}
Let $u \in L^{p^*}(R^n)$ be a weak solution of (\ref{1}). Then
$\nabla u \in L^p(R^n)$ if and only if $|x|^au^{q+1}(x) \in
L^1(R^n)$. Here $p^*=\frac{np}{n-p}$.
\end{theorem}

\begin{proof}
Choose a smooth function $\zeta(x)$ satisfying
$$
\left \{
   \begin{array}{lll}
&\zeta(x)=1, \quad &for~ |x| \leq 1;\\
&\zeta(x) \in [0,1], \quad &for~ |x| \in [1,2];\\
&\zeta(x)=0, \quad &for~ |x| \geq 2.
   \end{array}
   \right.
$$
Take the test function in (\ref{weaksolution}) as
\begin{equation} \label{cut}
\zeta_R(x)=\zeta(\frac{x}{R}).
\end{equation}
Thus for $D:=B_{3R}$, there holds
\begin{equation} \label{L1}
\int_D A(x,\nabla u) \nabla u \zeta_R^{p}dx+p\int_D u\zeta_R^{p-1}
A(x,\nabla u) \nabla\zeta_R dx =\int_D |x|^au^{q+1}\zeta_R^{p}dx.
\end{equation}

{\it Necessity.} If $\nabla u \in L^p(R^n)$, we claim
$|x|^au^{q+1} \in L^1(R^n)$. In fact, by using (\ref{2}) and the
Young inequality, we get
\begin{equation} \label{Holder}
\begin{array}{ll}
&\quad |\displaystyle\int_D u\zeta_R^{p-1} A(x,\nabla u) \nabla\zeta_R
dx|\\[3mm]
&\leq C(\displaystyle\int_D|\nabla u|^p\zeta_R^pdx)^{1-1/p}
(\int_Du^{p^*}dx)^{1/p^*}(\int_D|\nabla\zeta_R|^ndx)^{1/n}\\[3mm]
&\leq \delta \displaystyle\int_D|\nabla u|^p\zeta_R^pdx+C
(\int_Du^{p^*}dx)^{p/p^*}(\int_D|\nabla\zeta_R|^ndx)^{p/n}
\end{array}
\end{equation}
for any $\delta \in (0,1/3)$.
Here $C>0$ is independent of $R$.
Inserting this into (\ref{L1}) and using (\ref{2}), we obtain
$$
\int_D |x|^au^{q+1}\zeta_R^{p}dx \leq C\int_D|\nabla u|^p\zeta_R^{p}dx
+C(\int_Du^{p^*}dx)^{p/p^*}(\int_D|\nabla\zeta_R|^ndx)^{p/n}.
$$
Letting $R \to \infty$ and noting
\begin{equation} \label{LL1}
\lim_{R \to \infty}\int_D|\nabla\zeta_R|^ndx <\infty,
\end{equation}
we get
$$
\int_{R^n} |x|^au^{q+1}dx <\infty.
$$

{\it Sufficiency.} If $u \in L^{p^*}(R^n)$ and $|x|^au^{q+1} \in
L^1(R^n)$, we claim $\nabla u \in L^p(R^n)$.

Inserting (\ref{Holder}) into (\ref{L1}), and taking $\delta$
sufficiently small, we deduce from (\ref{2}) that
$$
\int_D|\nabla u|^p\zeta_R^{p}dx \leq C\int_D
|x|^au^{q+1}\zeta_R^{p}dx
+C(\int_Du^{p^*}dx)^{p/p^*}(\int_D|\nabla\zeta_R|^ndx)^{p/n}.
$$
Letting $R \to \infty$ and using (\ref{LL1}) we also obtain
$$
\nabla u \in L^p(R^n).
$$
The proof of Theorem \ref{prop6.2} is complete.
\end{proof}

According to Theorem \ref{prop6.2}, we call a positive weak
solution $u$ is the {\it finite energy solution} of (\ref{1}), if
$u \in L^{p^*}(R^n)$ and $|x|^au^{q+1} \in L^1(R^n)$. By the
results in section 4, we also call a positive weak solution $u$ is
the {\it integrable solution} of (\ref{1}), if $u \in
L^{s_0}(R^n)$. When $q$ is equal to the critical exponent
$\frac{p(n+a)}{n-p} -1$, $p^*=s_0$ and hence the finite energy
solution is an integrable solution. On the contrary, we have the
following result (Theorem \ref{prop6.3}). This result, together
with Theorem \ref{prop6.2}, implies that
\begin{equation} \label{tiantian}
u \in L^{s_0}(R^n) \Rightarrow u \in
\mathcal{D}^{1,p}(R^n)
\end{equation}
as long as $u$ is a positive weak solution.

\begin{theorem} \label{prop6.3}
The integrable solution is also a finite energy solution.
\end{theorem}

\begin{proof}
If $u$ is an integrable solution, Theorem \ref{equiva} shows that
$u$ solves (\ref{IE}). According to Theorems \ref{inter} and
\ref{boundth}, we have $u \in L^s(R^n)$ for all $\frac{1}{s} \in
[0,\frac{n-p}{n(p-1)})$. Therefore, $u \in L^{p^*}(R^n)$. In
addition, by the H\"older inequality, we obtain
$$
\int_{B_1(0)}|x|^au^{q+1}(x)dx  \leq
\|u\|_{\infty}^{q+1}\int_{B_1(0)}|x|^{a}dx <\infty.
$$
and
$$
\int_{R^n\setminus B_1(0)}|x|^au^{q+1}(x)dx
\leq (\int_{R^n\setminus B_1(0)}|x|^{ak}dx)^{1/k}
\|u^{q+1}\|_{k'}<\infty
$$
by taking $\frac{1}{k}=\frac{-\epsilon-a}{n}$ and
$\frac{1}{k'}=1-\frac{1}{k}$ with $\epsilon>0$ sufficiently small.
Thus, $|x|^au^{q+1} \in L^1(R^n)$. Namely, $u$ is a finite energy
solution.
\end{proof}

Combining Theorems \ref{tiant}-\ref{prop6.3}, we complete the proof
of Theorem \ref{th1.1}.

Next, we prove Corollary \ref{th1.3}. The following result is needed.

\begin{theorem} \label{th4.8}
If a classical solution $u$ of (\ref{PDE})
belongs to $\mathcal{D}^{1,p}(R^n)$, then $q$ is the critical exponent $p_*-1$, and
$\|\nabla u\|_p^p=\||x|^au^{p_*}\|_1$. Here $p_*=\frac{p(n+a)}{n-p}$.
\end{theorem}

\begin{proof}
Write $B=B_R(0)$. First, multiplying by $u$ and integrating on
$B$, we have
\begin{equation} \label{guoo}
\int_B|x|^au^{q+1}dx=\int_B|\nabla u|^pdx -\int_{\partial
B}|\nabla u|^{p-2}u \partial_{\nu}u ds.
\end{equation}
Here $\nu$ is the unit outward normal vector to $\partial B$. By
virtue of $u \in \mathcal{D}^{1,p}(R^n)$, we can find $R_j \to
\infty$ such that
$$
R_j\int_{\partial B_j}(|\nabla u|^p+u^{p^*})ds \to 0,
$$
where $B_j=B(0,R_j)$. Hence, when $R_j \to \infty$,
$$\begin{array}{ll}
&\quad |\displaystyle\int_{\partial B_j}|\nabla u|^{p-2}u
\partial_{\nu}u ds|\\[3mm]
&\leq (\displaystyle\int_{\partial B_j}|\nabla
u|^pds)^{1-\frac{1}{p}} (\int_{\partial
B_j}u^{p^*}ds)^{\frac{1}{p^*}}|\partial B_j|^{\frac{1}{n}}\\[3mm]
&\leq C(R_j\displaystyle\int_{\partial B_j}|\nabla
u|^pds)^{1-\frac{1}{p}} (R_j\int_{\partial
B_j}u^{p^*}ds)^{\frac{1}{p^*}}
R_j^{\frac{n-1}{n}-1+\frac{1}{p}-\frac{1}{p^*}}\\[3mm]
&\to 0.
\end{array}
$$
Inserting this into (\ref{guoo}) with $R=R_j \to \infty$, we get
\begin{equation} \label{guo}
\|\nabla u\|_p^p=\||x|^au^{q+1}\|_1.
\end{equation}

Multiplying the equation with $(x\cdot \nabla u)$ and integrating
on $B$, we obtain
$$
\int_B|\nabla u|^{p-2}\nabla u\nabla(x\cdot\nabla u)dx
-\int_{\partial B}|\nabla u|^{p-2} \partial_\nu u (x\cdot\nabla
u)ds =\int_B |x|^au^p(x\cdot\nabla u)dx.
$$
Noting
$$
\nabla u\nabla(x\cdot\nabla u)=|\nabla
u|^2+\frac{1}{2}x\cdot\nabla(|\nabla u|^2)
$$
and $x=|x|\nu$, we have
$$\begin{array}{ll}
&\displaystyle\int_B|\nabla u|^p dx+\frac{1}{p}\int_B
x\cdot\nabla(|\nabla u|^p)dx
-R\int_{\partial B}|\nabla u|^{p-2}|\partial_\nu u|^2ds\\[3mm]
&=\displaystyle\frac{1}{q+1}\int_B |x|^a x\cdot\nabla u^{q+1}dx.
\end{array}
$$
Integrating by parts, we get
\begin{equation} \label{deng}
\begin{array}{ll}
&\displaystyle(1-\frac{n}{p})\int_B|\nabla u|^p dx
+\frac{R}{p}\int_{\partial B}|\nabla u|^p ds
-R\int_{\partial B}|\nabla u|^{p-2}|\partial_\nu u|^2 ds\\[3mm]
&=\displaystyle\frac{R^{1+a}}{q+1}\int_{\partial
B}u^{q+1}ds-\frac{n+a}{q+1}\int_B|x|^au^{q+1}dx.
\end{array}
\end{equation}
According to Theorem \ref{prop6.2}, $u \in
\mathcal{D}^{1,p}(R^n)$ implies $|x|^a u^{q+1} \in L^1(R^n)$ and
$\nabla u \in L^p(R^n)$. Therefore, we can find $R_j \to \infty$,
such that
$$
R_j\int_{\partial B_{R_j}}(|x|^au^{q+1}+|\nabla u|^p)ds \to 0.
$$
Let $R=R_j \to \infty$ in (\ref{deng}). By means of the result
above, we deduce that
$$
(1-\frac{n}{p})\int_{R^n}|\nabla u|^p dx
=-\frac{n+a}{q+1}\int_{R^n}|x|^au^{q+1}dx.
$$
Combining with (\ref{guo}), we get $q=p_*-1$. Inserting this
result into (\ref{guo}), we complete the proof of Theorem
\ref{th4.8}.
\end{proof}

\paragraph{Remark 4.2.} Theorem \ref{th4.8} shows that if $q$ is not
equal to the critical exponent $p_*-1$, then there does not exist any classical
solution in $\mathcal{D}^{1,p}(R^n)$. In view of (\ref{tiantian}),
there does not exist any classical solution in $L^{s_0}(R^n)$.

\vskip 5mm

\textbf{Proof of Corollary \ref{th1.3}.}

{\it Step 1.} Item 1 $\Leftrightarrow$ item 2:

First the positive solution $u$ of (\ref{PDE}) is a $\mathcal{A}$-superharmonic
function. In addition, the integrability and the decay property of $u$ in items 1 and 2
ensure $\inf_{R^n}u=0$. Similar to the proof of Theorem \ref{equiva}, we know
that $u$ solves (\ref{IE}) with $\beta=1$. According to the argument in section 3, we
can see the equivalence easily.

{\it Step 2.} Item 1 $\Leftrightarrow$ item 3:

If $u \in L^{s_0}(R^n)$, by Theorem \ref{prop6.3} we get $u \in L^{p^*}(R^n)$
and $|x|^au^{q+1} \in L^1(R^n)$. Thus, we can deduce $\nabla u \in L^p(R^n)$ from
Theorem \ref{prop6.2}. Therefore, $u \in \mathcal{D}^{1,p}(R^n)$.

On the contrary, if $u \in \mathcal{D}^{1,p}(R^n)$, then Theorem \ref{th4.8}
shows $q+1=p_*$. Thus, $s_0=p^*$, and hence $u \in \mathcal{D}^{1,p}(R^n)$ implies
$u \in L^{s_0}(R^n)$.

\subsection{Weak solutions in $\mathcal{D}^{1,p}(R^n)$ instead of $W^{1,p}(R^n)$}

Now, we explain that the weak bounded solution of (\ref{1}) can
not be defined in the space $L^p(R^n)$ if $n \leq p^2$.

If $u \in L^p(R^n)$ is a positive bounded solution of (\ref{1}),
then $\inf_{R^n} u=0$, and hence $u$ solves (\ref{IE}).

When $u$ is not integrable, according to the argument of Theorem
\ref{th5.1}, there exists $R>0$ such that as $|x|>R$,
$$
u(x) \geq c|x|^{-\frac{p+a}{q-p+1}-\epsilon}
$$
with sufficiently small $\epsilon>0$. Thus,
$$
\int_{R^n} u^p(x)dx \geq c\int_{R^n \setminus B_R(0)}
\frac{dx}{|x|^{p(\frac{p+a}{q-p+1}+\epsilon)}}.
$$
In view of $n \leq p^2$, we get
$$
\frac{(n+a)(p-1)}{n-p}
\geq \frac{p(p+a)}{n}+(p-1).
$$
Therefore, (\ref{cond}) implies
$q > \frac{p(p+a)}{n}+p-1.$
Hence, we obtain easily
$$
n \geq p(\frac{p+a}{q-p+1}+\epsilon)
$$
as long as $\epsilon$ is suitably small. Hence, $\|u\|_p=\infty$.
It is impossible.

When $u$ is integrable,
according to Theorems 3.4, there exists $R>0$ such that as
$|x| >R$,
$$
u(x)\geq c|x|^{-\frac{n-p}{p-1}}.
$$
Therefore, by $n \leq p^2$, we have $n \geq p\frac{n-p}{p-1}$,
and hence
$$
\int_{R^n} u^p(x) dx \geq c\int_{R^n \setminus B_R(0)}
\frac{dx}{|x|^{p\frac{n-p}{p-1}}} =\infty.
$$
This also contradicts with $u \in L^p(R^n)$.

\paragraph{Remark 4.3.}
In particular, if $q$ is equal to the critical exponent
$\frac{p(n+a)}{n-p}-1$, then $p^*=s_0$. By the Sobolev inequality,
$u \in W^{1,p}(R^n)$ implies that $u$ is integrable. According to
the argument above, we know that $u \not\in L^p(R^n)$.
Thus, we can only assume that the weak solution $u \in
\mathcal{D}^{1,p}(R^n)$ instead of $W^{1,p}(R^n)$ as long as $n
\leq p^2$.

\paragraph{Acknowledgements.} The research was supported by NSF (No. 11171158) of China,
the Natural Science Foundation of Jiangsu (No. BK2012846) and SRF for ROCS, SEM.

\end{document}